\documentclass[a4,12pt]{amsart}

\usepackage{amssymb,amstext,amsmath}
\usepackage{enumitem}
\usepackage{color}
\usepackage{comment}
\usepackage{graphicx}        

 \makeatletter
  \def\sol{\accentset{{\cc@style\overline{\mskip10mu}}}} 
  \makeatother
 \makeatletter
  \def\sul{\underaccent{{\cc@style\underline{\mskip10mu}}}}
  \makeatother 
\theoremstyle{definition}
\newtheorem{definition}{Definition}
\theoremstyle{plain}
\newtheorem{theorem}[definition]{Theorem}

\newtheorem{proposition}[definition]{Proposition}
\theoremstyle{remark}
\newtheorem{remark}[definition]{Remark}

\newtheorem{example}[definition]{Example}

\DeclareMathOperator\USC{USC}

\DeclareMathOperator\tr{tr}

\def\stm{\setminus}
\newcommand{\R}{\mathbb{R}}

\newcommand{\N}{\mathbb{N}}

\def\cS{\mathcal{S}}
\def\cM{\mathcal{M}}

\def\bye{\end{document}}
\def\by{\end{proof}\bye}
\def\fr{\frac} 
\def\disp{\displaystyle}  
\def\ga{\alpha}     

\def\gep{\varepsilon}      
\def\ep{\gep}    
\def\mid{\,:\,}

\def\gd{\delta}
\def\gz{\zeta}

\def\gl{\lambda}
\def\gL{\Lambda}

\def\tim{\times}

\def\ol{\overline}
           
\def\pl{\partial}

\newcommand{\Mpl}{\mathcal{M}^+_{\lambda, \Lambda}}

\def\bcases{\begin{cases}}
\def\ecases{\end{cases}}
\def\beq{\begin{equation}}
\def\eeq{\end{equation}}
\def\balns{\begin{align*}}
\def\ealns{\end{align*}}
\def\bald{\begin{aligned}}
\def\eald{\end{aligned}}
\def\bred{\begin{color}{red}} \def\ered{\end{color}}
\def\gO{\Omega} 

\def\1{\mathbf{1}}

\def\IN{\text{ in }}\def\IF{\text{ if }} \def\FOR{\text{ for }} 
\def\AND{\text{ and }}

\def\ds{\rightarrow}
\def\dis{\displaystyle}

\title{Test function approach to fully nonlinear equations in thin domains}

\author[I. Birindelli]{Isabeau Birindelli}
\address[I. Birindelli]{Dipartimento di Matematica Guido Castelnuovo, Sapienza 
Universit\`a di Roma, Piazzale Aldo Moro 5, Roma, Italy.}
\email{isabeau@mat.uniroma1.it}

\author[A. Briani]{Ariela Briani}
\address[A. Briani]{Institut Denis Poisson, 
Universit\'e de Tours, France.}
\email{ariela.briani@univ-tours.fr}

\author[H. Ishii]{Hitoshi Ishii
}
\address[H. Ishii]{Institute for Mathematics and Computer Science, Tsuda  University,
 2-1-1 Tsuda, Kodaira, Tokyo 187-8577 Japan.}
\email{hitoshi.ishii@waseda.jp}


\thanks{A. Briani was
partially supported by l’Agence Nationale de la Recherche (ANR), project
ANR-22-CE40-0010 COSS; I. Birindelli was partially supported by project Leoni 2023 GNAMPA-INDAM and project  "At the Edge of Reaction-diffusion equations" Sapienza Università di Roma. H. Ishii was partially supported by the JSPS KAKENHI Grant Nos. JP20K03688, JP20H01817 and JP21H00717. The project was very much advanced while I. Birindelli was visiting prof. Ishii in Tsuda University, she wished to thank the Institution for the invitation.}


\begin{document} 


\def\ON{\text{ on }}

\begin{abstract}
In this note we extend to fully nonlinear operators the well known result on thin domains of  Hale and Raugel \cite{HR}. The result is more general even in the case of the Laplacian.
\end{abstract}

\maketitle 

\tableofcontents
\allowdisplaybreaks

\section{Introduction} 
The classical result of Hale and Raugel \cite{HR} in thin domains states that if $u_\ep$ are solutions of
$$\left\{\begin{array}{lc}
-\Delta u_\ep+u_\ep=f(x,y) &\mbox{ in }\ \Omega_\ep\\
 \dis \frac{\partial u_\ep} { \partial \nu_\ep} =0 &\mbox{ on }\ \partial\Omega_\ep
\end{array}
\right.
$$
where $
\gO_\ep=\{(x,y)\in\R^N\times\R \mid x\in\Omega, \ 0<y<\ep g(x)\},
$ for some $g\in C^{3}(\overline{\Omega})$ such that ${\disp 0<\inf_{\Omega} g\leq \sup_{\Omega}g<\infty}$
 then $u_\ep$ converges to $u_o$ solution of
$$\left\{\begin{array}{lc}
-(\Delta u_o+\frac{D g\cdot Du_o}{g})+u_o=f(x,0) &\mbox{ in }\ \Omega\\
\dis \frac{\partial u_0}{\partial  \nu} =0 &\mbox{ on }\ \partial\Omega.
\end{array}
\right.
$$
This result has been extended in a wide variety of related problems see e.g. the works of Arrieta, Pereira, Raugel \cite{AP, ANP,R}.
But all the above results concern variational problems, where the appearance of the first order seems to come from a typical integration by part, related to the variational nature of the problem. 

In this paper, instead, we treat fully nonlinear equation in thin domains i.e.
where the equation is given by
$$F(D^2u,D u, u,(x,y))=0 \ \mbox{in }\Omega_\ep$$
where $F: \cS(N+1) \times \R^{N+1} \times \R \times \gO_\ep \ds \R$ is a proper functional in the sense of the User's guide \cite{CIL}.
Of course the solutions are viscosity solutions and the proof follows the test function approach of Evans \cite{Ev} which is somehow more direct and completely different from the papers mentioned above. Furthermore the technique does not require the operator to be uniformly elliptic as it will be evident from the hypotheses below. An example of thin domains for degenerate elliptic operator will be given explicitly below.  

Even though the results will be proved for a large class of operators, in this introduction, we will illustrate the special case 
where the fully non linear operator is one of the extremal Pucci operators e.g. for $0<\lambda\leq \Lambda$
$$\Mpl(D^2u):=\sup_{\lambda I\leq A\leq \Lambda I}(\tr A(D^2u)):=\lambda\sum_{e_i\leq 0} e_i +\Lambda\sum_{e_i\geq 0} e_i, $$
where $e_i=e_i(D^2u)$ denotes the $i$-th eigenvalue of the Hessian matrix $D^2u$. 

Under the hypothesis
\begin{enumerate}\item[(H1)] 
$g\in C^1(\overline\Omega)$ \ \  and \ \ 
${\disp 0<\inf_{\Omega} g\leq \sup_{\Omega}g<\infty,}$
\end{enumerate}
we will prove that  $u_\ep$, the solutions of
\beq\label{I1}\left\{\begin{array}{lc}
-\Mpl (D^2 u_\ep)+u_\ep=f(x,y) &\mbox{ in }\ \Omega_\ep\\
\dis \frac{\partial u_\ep}{\partial \nu_\ep} =0 &\mbox{ on }\ \partial\Omega_\ep
\end{array}
\right.
\eeq
converge uniformly to $u_o$ solution of
\beq\label{rem1.2}
\left\{\begin{array}{lc}
-\Mpl(D^2u_o(x))-\Lambda \left(\frac{D g\cdot Du_o}{g}\right)^++\lambda \left(\frac{D g\cdot Du_o}{g}\right)^-+ u_o(x)=f(x,0)
 & \IN \Omega\\
\dis \frac{ \partial u_o}{\partial \nu} =0 &\mbox{ on }\ \partial\Omega.
 \end{array}
 \right.
 \eeq
Of course in the first equation $\Mpl$ acts on matrices in $\cS(N+1)$ while, in the second equation, it acts on matrices in $\cS(N)$.

In the special case $\lambda=\Lambda=1$, when $\Mpl=\Delta$, we recover Hale and Raugel result, but we improve the condition on $g$ that is only required to be the natural condition $C^1$ and not $C^3$.

We wish to explain the {\bf heuristic} behind the formal proof which will be given in this paper, for a much larger class of operators. Let $u_\ep$ be a solution of \eqref{I1} and let 
$$v_\ep (x,y):=u_\ep(x, \ep g(x)y)$$ 
so that we have "flattened" the top boundary.
In similarity with the linear variational case we can suppose that there exists a constant $C$ such that
$$|\partial_{yy}v_\ep|\leq C\ep^2.$$
This in turn implies that for $\ep\rightarrow 0$, $\partial_{yy}v_\ep\rightarrow 0$ and then, using the boundary condition, we get
$$v_\ep(x,y)\rightarrow v_o(x).$$
On the other hand, the above estimates implies also that, for some function $k(x)$, we get that
$$\frac{\partial_{yy}v_\ep(x,y)}{\ep^2}\rightarrow k(x):=g^2(x)h(x).$$
So we may use the following \emph{ansatz}:
$$v_\ep(x,y)=w(x)+\ep^2k(x)\frac{y^2}{2} +o(\ep^2).$$

Substituting the ansatz in the equation, we let, formally $\ep$ go to zero, after a tedious but simple computation it is easy to see that we obtain
$$-\Mpl(\left(\begin{array}{cc}D^2w(x) & 0 \\
0 & h(x)\end{array}\right))+w(x) =f(x,0)$$

We use the condition on the "top" boundary, in order to determine $h(x)$:

$$Dg(x)\cdot[Dw(x)+\ep^2D(g^2 h)(x)\frac{1}{2}]=g(x)h(x)[1+\ep^2 Dg(x)].$$
Passing to the limit we find
$$h(x)=\frac{Dg(x)\cdot Dw(x)}{g}$$
i.e. the limit equation becomes:

$$-\Mpl(\left(\begin{array}{cc}D^2w(x) & 0 \\
0 & \frac{D g(x)\cdot D w(x)}{g(x)}\end{array}\right))+w(x) =f(x,0).$$
In the rest of the paper we will treat the general case and make rigorous the above idea.

We will first give some a priori bounds, which allow to prove that the upper and lower relaxed limits $u^+$ and $u^-$ of $\{u^\ep\}_{\ep\in(0,\ep_0]}$ are respectively sub and super solutions of the limit equation:
\beq 
\label{ieq2-2}\left\{\ \bald
&F\left(\begin{pmatrix}D^2 w&0 \\ 0& D g\cdot Dw/g \end{pmatrix}, (D w,0),w,(x,0)\right)=0 
\ \ \IN \Omega, 
\\&\fr {\pl w}{\pl \nu}=0 \ \ \ON \ \pl\gO.
\eald \right.
\eeq
Under the further condition that the comparison principle holds, we will prove that the convergence  of $u^\ep$ to the solution of \eqref{ieq2-2} is uniform.

The scope of this paper is to open the results in thin domains from a perspective that is, to our knowledge completely different from the previous results, see e.g. \cite{AP, ANP, HR, R}. So, in order to make the exposition the clearer possible, we have decided to concentrate on the more classical case i.e. domains like $\Omega_\ep$ that in one direction have a one flat boundary and with a Neumann boundary condition. We plan to investigate in a further research, more generale domains, for examples with jumps or with non flat sides, or thin domains that also have an oscillatory boundary.
We hope the reader will appreciate this choice.

\section{Preliminaries}
Let $F: \cS(N+1) \times \R^{N+1} \times \R \times \gO_\ep \ds \R$ be a proper functional in the sense of the User's guide \cite{CIL}, i.e. 
\\
\hspace{5pt} \begin{minipage}{0.92\textwidth}
\[\tag{H2}\left\{\bald
& F\in C(\cS(N+1) \times \R^{N+1} \times \R \times \gO_\ep,\R),  
\\& F(X,p,r,(x,y)) \leq F(Y,p,s,(x,y)) \mbox{ whenever } r \leq s , \mbox{ and } Y \leq X. 
\eald\right.
\] 
\end{minipage} \medskip 

Furthermore, for simplicity of the presentation, we strengthen the monotonicity 
condition on $F$ in the above as follows. 
\begin{enumerate}
\item[(H3)]There exists $\alpha>0$ such that 
\[\alpha(r-s)\leq F(X,p,r,(x,y))-F(X,p,s,(x,y))\]
for $r\geq s$ and $(X,p,(x,y))\in \cS(N+1)\times \R^{N+1}\times \gO_{\ep}.$
\end{enumerate}

Our PDE problem is:
\beq \label{eq1}
F(D^2 u^\ep, D u^\ep, u^\ep, (x,y))= 0 
\ \ \IN \gO_{\ep} \ \ \AND \ \ 
\fr {\pl u^\ep}{\pl 
\nu_\ep}=0 \ \ \ON \pl\gO_\ep,
\eeq
where $
\nu_\ep$ denotes the outward (unit) normal to $\gO_\ep$. 

Since our concern is the asymptotic behavior of solutions $u^\ep$ to \eqref{eq1}, 
we will restrict ourself to the parameter $\ep$ in the range $(0,\ep_0]$, where 
$\ep_0>0$ is a number fixed throughout (see the comment after Proposition \ref{prop1}).  
Let us note that we shall use $\nu$ to indicate the normal to $\gO$ and $\nu_\ep$ for $\gO_\ep$. 
Obviously, the assumptions above are not enough 
to ensure the existence of viscosity solutions in the sense of the User's guide \cite{CIL}  to \eqref{eq1}.  To keep the generality of the assumptions made above, we consider the notion 
of viscosity solutions to \eqref{eq1} which eliminates the continuity requirement. 
That is, we call a bounded function $u$ on $\ol{\gO_\ep}$ a (viscosity) solution 
of \eqref{eq1} if its upper and lower semicontinuous envelopes are 
viscosity sub and super solutions, in the sense of \cite{CIL}, to \eqref{eq1}, respectively. 

We assume throughout that 
\begin{enumerate} \item[(H4)] $\gO$ is a bounded $C^1$ domain of $\R^N$. 
\end{enumerate}
Accordingly, we may choose a function $\rho\in C^1(\R^N)$ so that 
\beq\label{rho}
\rho(x)<0  \ \ \FOR x\in\gO, \ \ D\rho(x)\not=0, \ \ \AND \ \ \rho(x)>0 \ \ \FOR x\in\R^N\stm \ol \gO.  
\eeq
Note that the outward unit normal $\nu$ to $\gO$ at $x\in \pl\gO$ is given by $\nu=|D\rho(x)|^{-1}D\rho(x)$. 
The domain $\gO_\ep$ has corners, where the $N$--dimensional hypersurface 
$\pl \gO\tim \R$ intersects either the hypersurfaces $y=g(x)$ or  $y=0$, respectively.  \\
We denote by $\pl_\mathrm{L}\gO_\ep$, 
$\pl_\mathrm{B}\gO_\ep$, and $\pl_\mathrm{T}\gO_\ep$ the lateral, bottom, and top portions of the boundary $\pl\gO_\ep$, which are described respectively as 
\[\bald
&\{(x,y)\in\pl\gO_\ep\mid x\in\pl\gO\}, \quad  \{(x,y)\in\pl\gO_\ep\mid y=0\}, \ \ 
\AND 
\\& \{(x,y)\in\pl\gO_\ep\mid y=\ep g(x)\}.
\eald
\] 

Furthermore, the outward unit normal to $\gO_\ep$ at the lateral boundary, 
at the bottom, $y=0$, and at the top boundary, $y=\ep g(x)$,  
is given, respectively, by $\nu_\mathrm {L}=(|D\rho(x)|^{-1}D\rho(x),0)$, $\nu_\mathrm{B}=-e_{N+1}=-(0,\ldots,0,1)$, and
\[
\nu_\mathrm{T}=\fr{(-\ep Dg(x),1)}{\sqrt{1+\ep^2|Dg(x)|^2}}. 
\]

The appearance of corners of the domain $\gO_\ep$ requires a little care in the definition of sub and super solutions to \eqref{eq1}. For instance, when $u$ is a bounded upper semicontinuous function on $\ol{\gO_\ep}$, we call $u$ a viscosity subsolution of 
\eqref{eq1} if the following condition holds: whenever $\phi\in C^2(\ol{\gO_\ep})$, $\hat z=(\hat x,\hat y)\in\ol{\gO_\ep}$ and 
$\max_{\ol{\gO_\ep}}(u-\phi)=(u-\phi)(\hat z)$, we have 
\beq\label{sub1}
F(D^2\phi(\hat z),D\phi(\hat z),u(\hat z),\hat z)\leq 0
\eeq
if $\hat z\in\gO_\ep$; we have either \eqref{sub1} or 
\beq \label{subL}
\nu_\mathrm{L}\cdot D\phi(\hat z)\leq 0
\eeq
if $\hat z\in\pl_\mathrm{L}\gO_{\ep}\stm(\pl_\mathrm{B}\gO_\ep \cup\pl_\mathrm{T}\gO_\ep)$;
we have either \eqref{sub1} or 
\beq \label{subB}
\nu_\mathrm{B}\cdot D\phi(\hat z)\leq 0 
\eeq if  $\hat z\in\pl_\mathrm{B}\gO_{\ep}\stm\pl_\mathrm{L}\gO_\ep$;
we have either \eqref{sub1} or
\beq
\label{subT}
\nu_\mathrm{T}\cdot D\phi(\hat z)\leq 0
\eeq
if $\hat z\in\pl_\mathrm{T}\gO_\ep\stm\pl_\mathrm{L}\gO_\ep$;
we have either \eqref{sub1}, \eqref{subL}, or \eqref{subB} if $\hat z\in\pl_\mathrm{L}\gO_\ep\cap\pl_\mathrm{B}\gO_\ep$. Finally we have either \eqref{sub1}, \eqref{subL}, or 
 \eqref{subT} if $\hat z\in \pl_\mathrm{L}\gO_\ep\cap\pl_\mathrm{T}\gO_\ep$.
Replacing ``$\max$'' and ``$\leq$'' with ``$\min$'' and ``$\geq $'', respectively, in the above condition yields the right definition of viscosity supersolution.

Thanks to (H2), we can fix a constant $C_0>0$ so that
\[
|F(0,0,0,(x,y))|\leq C_0 \ \ \FOR (x,y)\in\ol{\gO_{\ep_0}}.
\]

\begin{proposition}\label{prop1} Assume that (H1)--(H4) hold. \begin{enumerate} 
\item 
The constant functions $\ga^{-1}C_0$ and $-\ga^{-1}C_0$ 
are classical super and sub solutions to \eqref{eq1}, respectively. 
\item  
There is a 
viscosity solution to \eqref{eq1}. 
\item  
There is a constant $\ep_0>0$ such that if $0<\ep\leq\ep_0$,  then any viscosity solution $u$ to \eqref{eq1} 
satisfies \ $\sup_{\ol{\gO_{\ep}}}|u|\leq \ga^{-1}C_0$. 
\end{enumerate}
\end{proposition}

Henceforth, we will assume that $\ep_0$ is small enough so that if $0<\ep\leq \ep_0$, then 
$\sup_{\ol\gO_\ep}|u|\leq \ga^{-1}C_0$ for any solution $u$ to \eqref{eq1}.

\begin{proof} (1) Set $u(z)=\ga^{-1}C_0$ for $z\in\ol{\gO_\ep}$. It is clear that 
$\nu_\mathrm{L}\cdot D u(z)=0$ for $z\in\pl_\mathrm{L}\gO_\ep$, 
$\nu_\mathrm{B}\cdot Du(z)=0$ for $z\in\pl_\mathrm{B}\gO_\ep$, 
and $\nu_\mathrm{T}\cdot D\phi(z)=0$ for $z\in\pl_\mathrm{T}\gO_\ep$. 
It follows that 
\[\bald
F(D^2u(z),Du(z),u(z),z)&=F(0,0,u(z),z)
\geq F(0,0,0,z)+\ga u(z)\\&\geq -C_0+C_0=0 \ \ \FOR z\in\ol{\gO_\ep}. 
\eald\]
Thus, the constant function $\ga^{-1}C_0$ is a classical supersolution of \eqref{eq1}. 
Similarly, the constant function $-\ga^{-1}C_0$ is a classical subsolution of \eqref{eq1}.

(2) Thanks to (H2), the constant functions $\ga^{-1}C_0$ and
$-\ga^{-1}C_0$ are viscosity sub and super solutions of \eqref{eq1}, respectively. 
Hence, the Perron method readily yields a solution to \eqref{eq1}. Indeed, if we set 
\[\bald
u(z)=\sup\{v(z)\mid & v \text{ is a viscosity subsolution of \eqref{eq1},} 
\\& -\ga^{-1}C_0\leq v\leq \ga^{-1} C_0 \ \ON\ \ol{\gO_\ep}\} \ \ \FOR z\in\ol{\gO_\ep},
\eald
\] 
then the function $u$ is a viscosity solution to \eqref{eq1}. 

(3) We claim that if $\ep_0$ is sufficiently small, then there is a function 
$\psi\in C^2(\ol{\gO_\ep}) $ such that 
\beq\label{psi}
\nu\cdot D\psi(z)>0  \ \ 
\FOR \bcases
\nu=\nu_\mathrm{L} \AND z\in\pl_\mathrm{L}\gO_\ep, \\[3pt]
\nu=\nu_\mathrm{B} \AND z\in\pl_\mathrm{B}\gO_\ep,\\[3pt]
\nu=\nu_\mathrm{T} \AND z\in\pl_\mathrm{T}\gO_\ep.
\ecases
\eeq
We will come back to the proof of \eqref{psi} later and now assume the existence of such a $\psi$, and complete the proof of (3). 
Let $u$ be any viscosity solution of \eqref{eq1}. 
Let $v$ and $w$  be the upper and lower semicontinuous envelopes of $u$ on $\ol{\gO_\ep}$, respectively.  Fix any $\gd>0$. 
We prove by contradiction that $v\leq\ga^{-1}C_0$ on $\ol{\gO_\ep}$. Thus, we suppose that $\max_{\ol{\gO_\ep}} v>\ga^{-1}C_0$. Choosing positive constants $\gd$ and $\gamma$ small enough, we have 
$\max_{\ol{\gO_\ep}}(v-\gamma\psi)>\gd+\ga^{-1}C_0$. Set $\phi=\gamma\psi+\gd+\ga^{-1}C_0$ on $\ol{\gO_{\ep}}$. 
Let $\hat z\in\ol{\gO_\ep}$ be a maximum point of the function $v-\phi$. 
Noting that \eqref{psi} holds with $\phi$ in place of $\psi$, we find by the subsolution 
property of $v$ that
\[\bald
0&\geq F(D^2\phi(\hat z), D\phi(\hat z), v(\hat z),\hat z)
\\&\geq 
F(\gamma D^2\psi(\hat z),\gamma D\psi(\hat z), \gamma \psi(\hat z)+\ga^{-1}C_0,\hat z)+\ga \gd. 
\eald
\]
Sending $\gamma \to 0^+$, we obtain $F(0,0,\ga^{-1}C_0,\tilde z)+\ga\gd\leq 0$ 
for some $\tilde z\in\ol{\gO_\ep}$, which contradicts that $\ga^{-1}C_0$ is a 
classical supersolution of \eqref{eq1}. Hence, we conclude that $u\leq v\leq \ga^{-1}C_0$ on $\ol{\gO_\ep}$. A parallel argument ensures that $u\geq w\geq -\ga^{-1}C_0$ 
on $\ol{\gO_\ep}$. 

It remains to prove that there is $\ep_0>0$ such that for each $\ep\in(0,\ep_0]$, there exists $\psi=\psi_\ep\in C^2(\ol{\gO_\ep})$ for which \eqref{psi} holds. 
Fix any $\ep>0$.  Let $\rho\in C^1(\R^{N})$ be a function satisfying \eqref{rho}. 
Choose a function $\eta\in C^1(\R)$ such that 
\[
\eta(r)=0 \ \ \FOR r\leq -g_0 \ \ \AND \ \ 0<\eta'(r) \leq 1 \ \ \FOR r>-g_0,
\]
where $g_0:=\inf_{\gO}g>0$.  
For a positive constant $\gd$, we define $\psi=\psi_\ep$ 
on $\R^{N+1}$ by setting 
\[
\psi(x,y)=\gd\rho(x)+\ep\Big(\eta\Big(-\fr{y}{\ep}\Big)+\eta\Big(\fr{y-\ep g(x)}{\ep}\Big)\Big) \ \ \FOR (x,y)\in\R^N\tim\R. 
\] 
If $z=(x,y)\in \pl_\mathrm{L}\gO_\ep$, then 
\[\bald
\nu_\mathrm{L}\cdot D\psi(z)&=
(|D\rho|^{-1}D\rho,0) \cdot \Big(\gd (D\rho,0)- \eta'\Big(-\fr{y}{\ep}\Big)e_{N+1} 
\\&\quad +\eta'\Big(\fr{y-\ep g(x)}{\ep}\Big)(-\ep Dg(x),1)\Big)
\geq \gd|D\rho(x)| -\ep_0 |Dg(x)|.
\eald
\]
Similarly, if $z=(x,0)\in\pl_\mathrm{B}\gO_\ep$, then 
\[
\nu_\mathrm{B}\cdot D\psi(z)=-e_{N+1}\cdot D\psi(z)
=\eta'(0)-\eta'(-g(x))= \eta'(0),
\] 
and if $z=(x,y)\in\pl_\mathrm{T}\gO_\ep$, then
\[\bald
\nu_\mathrm{T}\cdot D\psi(z)&\geq \fr{1}{\sqrt{\ep^2|Dg|^2+1}}\Big(-\gd\ep |D\rho||Dg| 
-\eta'(-g(x)) +\eta'(0)(\ep^2 |Dg(x)|^2+1)) \Big)
\\&\geq \fr{1}{\sqrt{\ep^2|Dg|^2+1}}\Big(-\gd\ep |D\rho||Dg| 
+\eta'(0)(\ep^2 |Dg(x)|^2+1)) \Big)
\\&\geq -\gd|D\rho(x)|+\eta'(0). 
\eald
\] 
We select $\gd>0$ and $\ep_0>0$ sufficiently small so that 
\[
\eta'(0)>\gd\max_{\ol\gO}|D\rho| \ \ \AND \ \ \gd \min_{\pl\gO}|D\rho| >\ep_0\max_{\pl\gO}|Dg|, 
\]
which assures that \eqref{psi} holds if $\ep\in(0,\ep_0]$. At this point, 
we only have the $C^1$--regularity of $\psi$, but the standard mollification 
procedure provides a $C^2$--function $\psi$ which satisfies \eqref{psi} 
as far as $\ep\in(0,\ep_0]$.  \end{proof}

\section{Convergence results}
\subsection{Relaxed limits}
Let $u^\ep$ be a solution of \eqref{eq1}. By the choice of $\ep_0$ (see also Proposition \ref{prop1}),  
we have
\[
\|u^\ep\|_\infty\leq \frac{C_0}{\alpha}.
\]
This allows us to define the upper and lower relaxed limits $u^+$ and $u^-$ of $\{u^\ep\}_{\ep\in(0,\ep_0]}$:
\beq\label{u+-}\left\{\ \bald
& u^+(x)=\lim_{r\to 0^+} \sup\{u^\ep(\xi,\eta)\mid (\xi,\eta)\in\ol\gO_\ep, |\xi-x|<r, 0<\ep<r \},
\\&u^-(x)=\lim_{r\to 0^+}\inf\{u^\ep(\xi,\eta)\mid (\xi,\eta)\in\ol\gO_\ep, |\xi-x|<r, 
0<\ep<r\}. 
\eald\right.
\eeq
It follows that $u^+\geq u^-$ on $\overline\Omega$ and $u^+,-u^-\in\USC(\overline\Omega)$. 
The limit equation will be
\beq 
\label{eq2-2}\left\{\ \bald
&F\left(\begin{pmatrix}D^2 w&0 \\ 0& D g\cdot Dw/g \end{pmatrix}, (D w,0),w,(x,0)\right)=0 
\ \ \IN \Omega, 
\\&\fr {\pl w}{\pl \nu}=0 \ \ \ON \ \pl\gO.
\eald \right.
\eeq

 
 \begin{theorem}\label{prop2} Suppose that (H1)--(H4) hold. 
The functions $u^+$ and $u^-$ are,  respectively, sub and super solutions of \eqref{eq2-2}. 
\end{theorem}

\begin{proof} We treat only the subsolution property. By replacing $u^\ep$ by its 
upper semicontinuous envelope, we may assume that $u^\ep$ is upper semicontinuous 
on $\ol{\gO_\ep}$. 
Let $\phi\in C^2(\overline\Omega)$ 
and assume that for some $\hat x\in\overline\Omega$, 
\[
(u^+-\phi)(x)<(u^+-\phi)(\hat x) \ \ \IF \ x\not=\hat x. 
\]

In the following computation, we fix $\gd>0$ arbitrarily.  
We choose $h_\gd\in C^2(\overline\Omega)$ so that 
\[
\left|\left(\fr{Dg\cdot D\phi}{g}\right)(x)-h_\gd(x)\right|<\gd \ \ \FOR x\in \overline\Omega.
\]

We set 
\[
\psi_\gd^\pm(x,y)=\fr{1}{2}y^2\left(\pm 2\gd+h_\gd(x)\right),
\]
and consider the function 
\[
\Phi(x,y)=\phi(x)+\psi_\gd^+(x,y)+\gamma \ep^2 \gz(y/\ep),
\]
where $\gz\in C^2(\R)$ is a bounded function on $\R$ having the properties 
\[-1<\gz'(0)<0<\gz'(y)<1 \ \FOR y\geq \min g \ \ \AND
\ \  |\gz''(y)|<1\ \FOR y\in [0,\max g], 
\]
and $\gamma>0$.

We choose a maximum point $(\bar x,\bar y)=(\bar x(\ep,\gamma),
\bar y(\ep,\gamma))$ of the function $u^\ep-\Phi$ on $\ol\gO_\ep$. 
We are to take the limit $\ep\to 0^+$. In our limit process as $\ep \to 0^+$, the choice 
of $\gamma$ depends on $\ep$ in such a way that $\lim \gamma/\ep=0$. 
A possible choice is $\gamma=\ep^2$.    
It is a standard observation (see Remark \ref{rem3} below) that as $\ep\to 0^+$, 
\beq \label{p2.1} 
(\bar x, \bar y) \to (\hat x, 0) \ \ \AND \ \ u^\ep(\bar x,\bar y) \to u^+(\hat x).
\eeq

Since $u^\ep$ is a subsolution of (1), if 
\[\tag{i}
(\bar x,\bar y)\in\gO_\ep,
\] then we have 
\beq 
\label{p2.2}
F(D^2 \Phi(\bar x,\bar y), D \Phi (\bar x,\bar y), u^\ep(\bar x,\bar y), (\bar x,\bar y) ) \leq 0; 
\eeq
if 
\[\tag{ii} (\bar x,\bar y)\in\pl_\mathrm{T}\gO_\ep\stm \pl_\mathrm{L}\gO_\ep, 
\] 
then we have either \eqref{p2.2} or 
\beq \label{p2.3} 
-\ep Dg(\bar x)\cdot D_x\Phi(\bar x, \ep g(\bar x))+\Phi_y(\bar x,\ep g(\bar x))\leq 0;
\eeq
if 
\[\tag{iii} (\bar x,\bar y)\in\pl_\mathrm{B}\gO_\ep\stm\pl_\mathrm{L}\gO_\ep, 
\]
then we have either \eqref{p2.2} or 
\beq\label{p2.4}  
-\Phi_y(\bar x, 0) \leq 0; 
\eeq
if 
\[\tag{iv} (\bar x,\bar y)\in\pl_\mathrm{L}\gO_\ep \stm(\pl_\mathrm{T}\gO_\ep \cup\pl_\mathrm{B}\gO_\ep),
\]
then we have either \eqref{p2.2} or
\beq \label{p2.5}  
\fr {\pl \Phi(\bar x,\bar y) }{\pl \nu_\ep}
=\nu_\mathrm{L}\cdot D\Phi(\bar x,\bar y)\leq 0;
\eeq
if 
\[\tag{v} (\bar x,\bar y)\in\pl_\mathrm{T}\gO_\ep\cap \pl_\mathrm{L}\gO_\ep,
\]
then we have either \eqref{p2.2}, \eqref{p2.3}, or \eqref{p2.5};  
if 
\[\tag{vi} (\bar x,\bar y)\in\pl_\mathrm{B}\gO_\ep\cap\pl_\mathrm{L}\gO_\ep,
\]
then we have either \eqref{p2.2}, \eqref{p2.4}, or \eqref{p2.5}. 

Observe that
\[\bald
D_x \Phi&=D \phi(x)+\fr{y^2}{2}D h_\gd(x),\quad
 &\Phi_y&=y\left(2\gd+h_\gd(x)\right)+\ep\gamma\gz'\left(\fr{y}{\ep}\right),
\\ D_{x}^2 \Phi &=D^2\phi(x)+\fr{y^2}2 D^2h_\gd(x),
\quad &\Phi_{yy}&=2\gd+h_\gd(x)+\gamma\gz''\left(\fr y{\ep}\right),
\\ \Phi_{x_iy}&=\Phi_{yx_i}= y\, (h_{\gd})_{x_i}(x).
\eald 
\]

 Inequalities \eqref{p2.2}, \eqref{p2.3}, \eqref{p2.4}, and \eqref{p2.5},  can be written, respectively, as
\beq \label{p2.6}
F(\bar{X},\bar{p},u^\ep(\bar x,\bar y),(\bar x,\bar y))\leq 0, 
\eeq
where 
$$
\bar{X}=\begin{pmatrix}   \disp D^2\phi(\bar x)+\fr{\bar y^2}2 D^2h_\gd( \bar x)  & \bar y Dh_\gd(\bar x) \\ \bar y D h_\gd(\bar x)^T  &\disp  2\gd+h_\gd(\bar x)+\gamma\gz''\left(\fr {\bar  y}{\ep}\right) \end{pmatrix},
$$
and \ $\disp 
\bar{p}=( D\phi(\bar x)+\fr{\bar y^2}{2}Dh_\gd(\bar x), {\bar y}\left(2\gd+h_\gd(\bar x)\right)+\ep\gamma\gz'\left(\fr{ \bar y}{\ep}\right) ), $

\beq \label{p2.7}
\bald
-D g(\bar x)\cdot&\left(D\phi(\bar x)+\fr{\ep^2 g(\bar x)^2}{2}Dh_\gd(\bar x)\right)
\\& +g(\bar x)\left(2\gd+h_\gd(\bar x)\right) +\gamma\gz'(g(\bar x))\leq 0,
\eald
\eeq
\beq \label{p2.8}  
\gz'\left(0\right)\geq 0,
\eeq
\beq \label{p2.9} 
\fr {\pl \Phi(\bar x,\bar y) }{\pl \nu_\mathrm{L}}= (|D\rho|^{-1}D\rho(\bar x),0)\cdot  D \Phi(\bar x,\bar y) \leq 0. 
\eeq

Choosing $\ep>0$ small enough, we may assume that
\[
\gd\geq \fr{\ep^2 g(\bar x)}{2}Dg(\bar x)\cdot Dh_\gd(\bar x). 
\]
If \eqref{p2.7} holds, then we have
\[\bald
0&\geq -Dg(\bar x)\cdot\Big(D\phi(\bar x) 
 + \fr{\ep^2 g(\bar x)^2}{2}Dh_\gd(\bar x)\Big)
\\&\quad 
+g(\bar x)\left(\gd+\left(\fr{Dg\cdot D \phi}{g}\right)(\bar x)\right)
+\gamma\gz'(g(\bar x))
\\&\geq \gamma \gz'(g(\bar x)). 
\eald
\]
This contradicts our choice of $\gz$, and also \eqref{p2.8} is a contradiction. 

Thus, we have \eqref{p2.6} in the case when either (i), (ii), or (iii) is valid, and
we have either \eqref{p2.6} or \eqref{p2.9} in the cases when either (iv), (v), or (vi) holds.

Sending $\ep\to 0^+$, we have 
\[
\bar{X} \to
\begin{pmatrix}   \disp D^2\phi(\hat  x)  & 0 \\ 0 &\disp  2\gd+h_\gd(\hat x) \end{pmatrix}  \leq \begin{pmatrix}   \disp D^2\phi(\hat  x)  & 0 \\ 0 &\disp  3\gd+\fr{Dg(\hat x) \cdot D\phi(\hat x)}{g(\hat x)} \end{pmatrix} 
\]
and 
\[
\disp \bar p 
\to ( D \phi(\hat x), 0).
\]
Therefore, we see that if $\hat x\in \Omega$, then we have
\beq \label{p2.10} 
F\left(\begin{pmatrix}   \disp D^2\phi(\hat  x)  & 0 \\ 0 &\disp  3\gd+\fr{D g(\hat x)D\phi(\hat x)}{g(\hat x)}\end{pmatrix},  ( D\phi(\hat x), 0), u^+(\hat x), (\hat x,0)  \right) \leq 0, 
\eeq 
if $\hat x\in\partial\Omega$, then we have either \eqref{p2.10} 
or 
\[
\nu_{\mathrm{L}}\cdot ( D \phi(\hat x), 0) = \fr{\pl \phi(\hat x)}{\pl \nu}\leq 0.
\]

This guarantees that $u^+$ is a subsolution of \eqref{eq2-2}. 

A remark on the proof of the supersolution property of $u^-$ is that, in this case, one should use the perturbed test function 
\[
\Phi(x,y)=\phi(x)+\psi_\gd^-(x,y)-\gamma \ep^2\gz\left(\fr y\ep\right). \qedhere
\]
\end{proof} 

\begin{remark} \label{rem3} For a general approach to the proof of \eqref{p2.1}, 
we may refer to
the User's guide \cite{CIL}.  
Here, for the reader's convenience, we give a straightforward proof of 
\eqref{p2.1}.  
By the definition of $u^+(\hat x)$, we may choose $\{(\ep_j, x_j,y_j)\}_{j\in\N}$ so that 
\[
\ep_j \to 0^+,\quad (x_j,y_j)\in\ol{\gO_{\ep_j}},\quad 
|x_j-\hat x|<\fr 1j, \quad 
u^+(\hat x)<\fr 1j +u^{\ep_j}(x_j,y_j). 
\]
Since $\Phi$ depends on $\ep$, we write $\Phi_\ep$ for $\Phi$. 
Also, we write $(\bar x_j,\bar y_j)$ for $(\bar x,\bar y)$ with $\ep=\ep_j$. 
Thus, $(\bar x_j,\bar y_j)$ is a maximum point of $u^{\ep_j}-\Phi_{\ep_j}$, and 
we have 
\[
(u^{\ep_j}-\Phi_{\ep_j})(\bar x_j, \bar y_j) \geq (u^{\ep_j}-\Phi_{\ep_j})(x_j, y_j)
>-\fr 1j +u^+(\hat x) -\Phi_{\ep_j}(x_j,y_j). 
\]
We may assume after passing to a subsequence that for some $\tilde x\in\overline\Omega$ and 
$\tilde u\in\R$,
\[
\lim (\bar x_j, \bar y_j) = (\tilde x,0) \ \ \AND \ \ \lim u^{\ep_j}(\bar x_j, \bar y_j)=\tilde u. 
\] 
Since $(\ep_j,\bar x_j,\bar y_j) \to (0,\tilde x, 0)$, we see, by the definition of $u^+(\tilde x)$, that
\[
u^+(\tilde x)\geq \lim u^{\ep_j}(\bar x_j,\bar y_j)=\tilde u. 
\]
All the above together, we see in the limit as $j\to \infty$ that
\[
u^+(\tilde x)-\Phi_0(\tilde x,0) \geq \tilde u-\Phi_0(\tilde x,0)\geq u^+(\hat x)-\Phi_0(\hat x,0),
\]
where $\Phi_0(x,y):=\lim_{\ep\to 0^+}\Phi_\ep(x,y)=\phi(x)+\psi_\gd^-(x,y)$, that is, 
\[
(u^+-\phi)(\tilde x) \geq \tilde u-\phi(\tilde x)\geq (u^+-\phi)(\hat x),
\]
which shows that $\tilde x=\hat x$ and $\tilde u=u^+(\hat x)$. \hfill $\Box$
\end{remark}

\subsection{Uniform convergence}

Let $F$, $\gO$, and $g$ be as in the previous section. 
Define the function $G\in C(\cS(N)\tim\R^N\tim\R\tim\ol\gO,\R)$ by
\beq\label{eq2.0} 
G(X,p,r,x)=F\Big(\begin{pmatrix}X&0\\ 0& Dg(x)\cdot p/g(x) \end{pmatrix}, (p,0), r, (x,0)\Big). 
\eeq
Recall that the limit equation \eqref{eq2-2} for $u$ is stated as
\beq\label{eq2.1}
G(D^2u,Du,u,x)=0 \ \ \IN \gO \ \ \AND \ \ \fr{\pl u}{\pl\nu}=0 \ \ \ON \pl\gO. 
\eeq 

A convenient assumption for Theorem \ref{prop3} to draw a uniform convergence result is the validity of the comparison principle for \eqref{eq2.1}:   
\begin{enumerate}\item[(H5)]If $v$ and $w$ are viscosity sub and super solutions to \eqref{eq2.1}, respectively, then\ $v\leq w$ \ on $\ol\gO$.  
\end{enumerate}

Indeed, we have 
\begin{theorem}\label{prop3} Assume (H1)--(H5). Let $u^\ep$ be a 
viscosity solution to \eqref{eq1} for $\ep\in(0,\ep_0]$. Then, for the unique 
continuous viscosity solution $u^0$ of \eqref{eq2.1}, we have 
\beq\label{eq2.2}
\lim_{\ep\to 0^+}\max_{(x,y)\in\ol{\gO_\ep}}|u^\ep(x,y)-u^0(x)|=0. 
\eeq
\end{theorem}

\begin{proof} The following argument is standard in the asymptotic analysis based on the half-relaxed limits, but we here present it for the reader's convenience.  
Let $u^+$ and $u^-$ be the functions defined by \eqref{u+-}. By the definition, we have $u^-\leq u^+$ on $\ol\gO$ and $u^+,-u^-\in\USC(\ol\gO)$. Theorem \ref{prop2} ensures that $u^+$ and $u^-$ are viscosity 
sub and super solutions to \eqref{eq2.1}, respectively. Furthermore, (H5) assures that 
$u^+\leq u^-$ on $\ol\gO$. Hence, we see that $u^+=u^-$ on $\ol\gO$, which readily 
shows that $u^+=u^-$ is continuous on $\ol\gO$. Writing $u^0$ for $u^+=u^-$, 
we find that $u^0$ is a continuous viscosity solution to \eqref{eq2.1}.   

To check \eqref{eq2.2}, fix any $\gd>0$. By the definition of $u^+$, for any 
$x\in\ol\gO$, we select $r=r(\gd,x)>0$ so that 
\[
u^\ep(\xi,\eta)<u^0(x)+\gd \ \ \IF 0<\ep<r, \ (\xi,\eta)\in\ol{\gO_\ep}, \ \AND\ 
|\xi-x|<r.
\] 
Reselecting $r>0$ sufficiently smaller, we may assume that $u^0(x)<u^0(\xi)+\gd$ if $\xi\in\ol\gO$ 
and $|\xi-x|<r$.  Now, the above inequality can be stated as 
\beq\label{eq2.3}
u^\ep(\xi,\eta)<u^0(\xi)+2\gd \  \ \IF 0<\ep<r, \ (\xi,\eta)\in\ol{\gO_\ep}, \ \AND\ 
|\xi-x|<r. 
\eeq
Since $\ol\gO$ is compact, we can choose a finite number of balls, $B_1,\ldots,B_m$, 
which cover $\ol\gO$, such that for every $j\in\{1,\ldots,m\}$, if $x_j$ and $r_j$ denote, respectively,  the center and radius of $B_j$, then \eqref{eq2.3}, with $(x_j,r_j)$ in place of $(x,r)$, holds.   Setting $r_0=\min\{r_j\mid j=1,\ldots,m\}$, we find that  
\[
u^\ep(\xi,\eta)<u^0(\xi)+2\gd \  \ \FOR (\xi,\eta)\in\ol{\gO_\ep} \ \AND \ 
0<\ep<r_0.
\] 
An argument parallel to the above yields, after replacing $r_0>0$ by a smaller one if necessary,
\[
u^\ep(\xi,\eta)>u^0(\xi)-2\gd \  \ \FOR (\xi,\eta)\in\ol{\gO_\ep} \ \AND \ 
0<\ep<r_0,
\]    
which completes the proof of \eqref{eq2.2}. 
\end{proof}
Let us recall that there are a number of contests where the comparison principle (H5) holds. In particular, when dealing with Neumann boundary conditions, one can refer to the results of  Hitoshi Ishii \cite{Is91},  Guy Barles \cite{Ba93}  and Stefania Patrizi  \cite{P}.

We consider here the general comparison principle given in \cite[Theorem 7.5]{CIL}. This leads us to assume,  further hypotheses on  the domain $\gO$ and the operator $G$.

On $\gO$, in addition to (H4),  we need the uniform exterior 
sphere condition, i.e. that there is a constant $r_0>0$ such that 
\beq\label{eq2.4}
B_{r_0}(x+r_0\nu(x)) \cap \gO =\emptyset \ \ \FOR x\in\pl\gO,
\eeq  
where $B_r(x)$ denotes the open ball $\{y\in\R^N\mid |y-x|<r\}$. 
On the function $G$ a  crucial and typical hypothesis  is the following: 
\beq \label{eq2.5} \left\{\begin{minipage}{0.85\textwidth}
There is a function $\omega:[0,\infty)\rightarrow [0,\infty)$ that satisfies $\omega(0^+)=0$ such that 
$$ G(Y,p,r, y)- G(X,p,r, x)\leq \omega(\gamma |x-y|^2+|x-y|(|p|+1))$$
whenever $\gamma>0$, $p\in\R^N$, $x, y \in \ol{\gO}$, $r\in\R$, and $X,Y\in\cS(N)$ satisfy  
\[
-3\gamma I_{2N} \leq \begin{pmatrix}X&0\\ 0&-Y \end{pmatrix} \leq 3\gamma\begin{pmatrix} I_{N} &-I_{N}\\ -I_{N}&I_{N} \end{pmatrix}.  
\] 
\end{minipage} \right.
\eeq
Here $I_m$ denotes the identity matrix of order $m$. 
We impose another continuity 
condition on $G$, which states: 
\beq \label{eq2.6}\left\{\begin{minipage}{0.85\textwidth}
There is a neighborhood $V$ of $\pl\gO$, relative to $\ol\gO$, such that 
$$ G(X,p,r, x)- G(Y,q,r, x)\leq \omega(\|X-Y\|+|p-q|)$$
for $X,Y\in\cS(N)$, $p,q\in\R^N$, $r\in\R$, and $x\in V$.
\end{minipage} \right. 
\eeq
Note that if (H3) holds, then  
\beq \label{eq2.7} 
\ga(r-s)\leq G(X,p,r,x)-G(X,p,s,x)
\eeq
for $ r\geq s$ and $(X,p,x)\in\cS(N)\tim\R^N\tim\ol\gO$. 

The next proposition is a direct consequence of \cite[Theorem 7.5]{CIL} and Theorem 
\ref{prop3}. 

\begin{proposition}\label{prop4} Assume (H1)--(H4) and \eqref{eq2.4}--\eqref{eq2.6}. 
Then (H5) is satisfied and the uniform convergence \eqref{eq2.2} as in Theorem 
\ref{prop3} is valid. 
\end{proposition}

Before concluding our discussion, we present two important examples of equations to which  
Theorem \ref{prop3} applies, one is fully nonlinear and the other is linear but degenerate elliptic.

\begin{example} We apply Proposition \ref{prop4} to show the uniform convergence result for the solution of equation \eqref{I1}, involving the extremal Pucci operator as presented in the Introduction. 
The extremal Pucci operator $-\cM^+_{\gl,\gL}(X)$ 
has the property \eqref{eq2.5}. Indeed, the matrix inequality on the right-hand side of \eqref{eq2.5} 
implies that $X\leq Y$ and hence, $-\cM^+_{\gl,\gL}(Y)+\cM^+_{\gl,\gL}(X)\leq 0$. 
If the regularity of $g$ is strengthened so that $g\in C^{1,1}(\ol\gO)$, then both the functions 
\[
H(p,x)=\left(\fr{Dg(x)\cdot p}{g(x)}\right)^\pm 
\] 
satisfy
\[
|H(p,y)-H(p,x)|\leq C|x-y||p|
\]
for all $p\in\R^N$, $x,y\in\ol\gO$ and some constant $C>0$.  It is then obvious to see 
that 
the operator 
\[
G(X,p,r,x)=-\cM^+_{\gl,\gL}(X)-\gL\left(\fr{Dg(x)\cdot p}{g(x)}\right)^++\gl\left(\fr{Dg(x)\cdot p}{g(x)}\right)^-+\ga r-f(x,0),
\]
where $f\in C(\ol{\gO_{\ep_0}})$, satisfies \eqref{eq2.5}.   Thus, thanks to Proposition \ref{prop4}, we find that 
the uniform convergence \eqref{eq2.2} for the solution $u^\ep$ to \eqref{I1}, as in Theorem \ref{prop3}, holds, provided that 
$\ga>0$, (H1), $g\in C^{1,1}(\ol{\gO})$, $f\in C(\ol{\gO_{\ep_0}})$, (H4), and \eqref{eq2.4} are satisfied.  
Of course the case of the Laplacian is recovered just by considering $\lambda=\Lambda=1$.
\end{example}

 \begin{example} In  these examples we concentrate on simple degenerate elliptic equations in order to emphasize how the nature of the limit equation depends on the direction of the diffusion.
Let $u_\ep$ be the solution of 
$$
- \frac{\partial^2  u_\ep} {\partial y^2}+u_\ep=f(x,y) \mbox{ in }\ \Omega_\ep,  \qquad  \frac{\partial  u_\ep}{ \partial \nu_\ep} =0 \mbox{ on }\ \partial\Omega_\ep. 
$$

If $g\in C^{1,1}(\ol{\gO})$ and $f\in C(\ol{\gO_{\ep_0}})$, then we are under the hypothesis of Proposition \ref{prop4}, therefore $u_\ep$ converges uniformly to $u_o$ solution of a first order equation precisely:
$$-\frac{D g\cdot Du_o}{g}+u_o=f(x,0) \mbox{ in }  \Omega,  \qquad \frac{\partial u_o}{\partial \nu} =0 \mbox{ on }\ \partial\Omega.  $$ 

 Instead, if  $u_\ep$ is the solution of 
 $$
- \frac{\partial^2  u_\ep} {\partial x_1^2}  +u_\ep=f(x,y) \mbox{ in } \:  \Omega_\ep, \qquad   \frac{\partial  u_\ep}{ \partial \nu_\ep}  \mbox{ on }\ \partial\Omega_\ep 
 $$

it will converge to  $u_o$ solution of the second order equation
$$ - \frac{\partial^2  u_0} {\partial x_1^2} +u_o=f(x,0)   \mbox{  in } \Omega, \qquad  \frac{\partial u_o}{\partial \nu}=0 \mbox{ on }\ \partial\Omega. $$

 \end{example}

\bye